\documentclass{article}
\usepackage{amsmath}
\usepackage{amssymb}
\usepackage{mathtools}
\usepackage{amsthm}
\usepackage{fancyhdr}
\usepackage{hyperref}
\theoremstyle{definition}
\newtheorem{definition}{Definition}[section]
\newtheorem{theorem}{Theorem}[section]
\newtheorem{proposition}{Proposition}[section]
\newtheorem{corollary}{Corollary}[theorem]
\newtheorem{lemma}{Lemma}[theorem]
\DeclareMathOperator{\E}{\mathbb{E}}
\DeclareMathOperator{\given}{\left.\right|}

\begin{document}

% Title of paper
\title{Exponential tilting of subweibull distributions}
\author{F. William Townes\\
        Department of Statistics and Data Science\\
        Carnegie Mellon University\\
        \texttt{ftownes@andrew.cmu.edu}
        }

\lhead{Townes 2024}
\rhead{Exponential tilting of subweibull distributions}

\maketitle

\tableofcontents

\section{Introduction}

Subexponential and subgaussian distributions are of fundamental importance in the application of high dimensional probability to machine learning \cite{vershyninHighDimensionalProbabilityIntroduction2018,wainwrightHighDimensionalStatisticsNonAsymptotic2019}. Recently it has been shown that the subweibull class unifies the subexponential and subgaussian families, while also incorporating distributions with heavier tails \cite{vladimirovaSubWeibullDistributionsGeneralizing2020,kuchibhotlaMovingSubGaussianityHighdimensional2022}. Informally, a q-subweibull ($q>0$) random variable has a survival function that decays at least as fast as $\exp(-\lambda x^q)$ for some $\lambda>0$. For example, the exponential distribution is 1-subweibull and the Gaussian distribution is 2-subweibull. Here, we provide two alternative characterizations of the subweibull class and introduce a distinction between strictly and broadly subweibull distributions. As an example, the Poisson distribution is shown to be strictly subexponential ($q=1$) but not subweibull for any $q>1$. Finally, we detail the conditions under which the subweibull property is preserved after exponential tilting.

\section{Laplace transforms}
\begin{definition}
\label{def:laplace-transform}
We define the two-sided Laplace transform of a random variable $X$ with distribution function $F$ as
\[\mathcal{L}_X(t)= \E[\exp(-tX)]=\int_{-\infty}^\infty \exp(-tx)dF(x)\]
\end{definition}
We do not restrict $X$ to be nonnegative or to have a density function. In the special case that $\mathcal{L}_X(t)<\infty$ for all $t$ in an open interval around $t=0$, then $X$ has a moment generating function (MGF) which is $M_X(t)= \E[\exp(tX)]=\mathcal{L}_X(-t)$. The Laplace transform can characterize the distribution even if the MGF does not exist.
\begin{lemma}
\label{lem:laplace-unique}
If the Laplace transforms of random variables $X$ and $Y$ satisfy $\mathcal{L}_X(t) = \mathcal{L}_Y(t)$ for all $t$ in any nonempty open interval $(a,b)\subset \mathbb{R}$, not necessarily containing zero, then $X\overset{d}{=} Y$. 
\end{lemma}
For a proof refer to \cite{mukherjeaNoteMomentGenerating2006}. A random variable $X$ is considered subexponential iff the MGF exists \cite{vershyninHighDimensionalProbabilityIntroduction2018}.  
If $\mathcal{L}_X(t)=\infty$ for all $t>0$ (respectively $t<0$), $X$ is said to have a heavy left (respectively, right) tail \cite{nairFundamentalsHeavyTails2022}. If a tail is not heavy it is said to be light. It is well known that the one-sided Laplace transform characterizes nonnegative distributions \cite{fellerIntroductionProbabilityTheory1971}. Lemma \ref{lem:laplace-unique} shows that the two-sided Laplace transform characterizes any distribution with at least one light tail. 

\section{Subweibull random variables}

\begin{definition}
\label{def:subweibull}
A random variable $X$ is \textit{q-subweibull} if $\E[\exp(\lambda^q \vert X\vert^q)]<\infty$ for some $\lambda>0$. $X$ is \textit{strictly q-subweibull} if the condition is satisfied for all $\lambda>0$. 
If $X$ is q-subweibull but not strictly so, we refer to it as \textit{broadly q-subweibull}. 
\end{definition}
The first part of this definition was also proposed by \cite{kuchibhotlaMovingSubGaussianityHighdimensional2022} and by \cite{vladimirovaSubWeibullDistributionsGeneralizing2020} using a parameterization equivalent to $1/q$. Clearly $X$ is (strictly) q-subweibull if and only if $\vert X\vert^q$ is (strictly) subexponential. As an example, the Laplace distribution is broadly 1-subweibull (ie broadly subexponential).

% Suppose $X$ is a random variable with distribution function $F$ and $p=\Pr(X<0)\notin \{0,1\}$. Define nonnegative random variables $A = [-X\given X<0]$ and $B = [X\given X\geq 0]$ with distribution functions $F^-$ and $F^+$, respectively. If $g$ is a nonnegative function then
% \begin{equation}
% \label{eq:gx-split}
% \begin{split}
% \E[g(X)] &= \E[g(X)\given X<0]p + \E[g(X)\given X\geq 0](1-p)\\
% &= \E[g(-A)]p + \E[g(B)](1-p)
% \end{split}
% \end{equation}

\begin{definition}
\label{def:radius-convergence}
The \textit{radius of convergence} of a q-subweibull random variable $X$ is defined by
\[R_q = \sup \left\{\lambda>0:\E[\exp(\lambda^q \vert X\vert^q)]<\infty\right\}\]
and if no such $\lambda>0$ exists we adopt the convention that $R_q=0$.
\end{definition}
In the case of strictly q-subweibull distributions, $R_q=\infty$. 
% Special cases include subexponential ($q=1$) and subgaussian ($q=2$) random variables. 
% See \cite{kuchibhotlaMovingSubGaussianityHighdimensional2022,vladimirovaSubWeibullDistributionsGeneralizing2020} for the case of $q<1$. 
$X$ has ``heavy tails'' (in the sense of \cite{nairFundamentalsHeavyTails2022}) iff it is not subexponential ($R_1=0$).
% We now generalize Propositions 2.5.2 and 2.7.1 of \cite{vershyninHighDimensionalProbabilityIntroduction2018} on subexponential and subgaussian random variables to the q-subweibull case. The proof is by a similar technique. 

\begin{lemma}
\label{lem:subweibull-split}
Random variable $X$ with $\Pr(X<0)\notin \{0,1\}$ is q-subweibull if and only if the nonnegative random variables $A = [-X\given X<0]$ and $B = [X\given X\geq 0]$ are q-subweibull. Let $R_{qx}$, $R_{qa}$, and $R_{qb}$ denote the radii of convergence for $X$, $A$, and $B$, respectively. Then $R_{qx}=\min\{R_{qa},R_{qb}\}$.
\end{lemma}
\begin{proof}
% Apply Equation \ref{eq:gx-split} with $g(x)=\exp(\lambda^q\vert x\vert^q)$. The left hand side is finite if and only if both terms on the right hand side are finite.
Let $p=\Pr(X<0)$ and define nonnegative random variables $A = [-X\given X<0]$ and $B = [X\given X\geq 0]$.
\begin{align*}
\E[\exp(\lambda^q\vert X\vert^q)] &= \E[\exp(\lambda^q (-X)^q)\given X<0]p + \E[\exp(\lambda^q X^q)\given X\geq 0](1-p)\\
&= \E[\exp(\lambda^q A^q)]p + \E[\exp(\lambda^q B^q)](1-p)
\end{align*}
The left hand side is finite if and only if both terms on the right hand side are finite. If $R_{qx}$ is the radius of convergence for $X$ then $\E[\exp(\lambda^q \vert X\vert^q)]<\infty$ for all $\lambda\in [0,R_{qx})$. Clearly $\E[\exp(\lambda^q A^q)]<\infty$ and $\E[\exp(\lambda^q B^q)]<\infty$ for all $\lambda\in [0,R_{qx})$ also, implying $\min\{R_{qa},R_{qb}\}\geq R_{qx}$. However, if $\min\{R_{qa},R_{qb}\}>R_{qx}$ then there exists some $\lambda> R_{qx}$ such that $\E[\exp(\lambda^q \vert X\vert^q)]<\infty$, which by Definition \ref{def:radius-convergence} means $R_{qx}$ is not the radius of convergence of $X$, a contradiction. Therefore, $\min\{R_{qa},R_{qb}\}=R_{qx}$. 
\end{proof}

\begin{proposition}(Theorem 1 of \cite{vladimirovaSubWeibullDistributionsGeneralizing2020})
\label{prop:subweibull}
The following are equivalent characterizations of a q-subweibull random variable $X$ for $q>0$. %$q\geq 1$.
\begin{enumerate}
  \item Tail bound: $\exists~K_{1a}>0$ such that $\forall~t\geq 0$, 
  \[\Pr(\vert X\vert > t) \leq 2\exp\big(-(t/K_{1a})^q\big)\] %1a
  % \begin{enumerate}[label=(\alph*)]
  %   \item $\exists~K_{1a}>0$ such that $\forall~t\geq 0$,
  %   \[\Pr(\vert X\vert > t) \leq 2\exp\big(-(t/K_{1a})^q\big)\] 
  %   \item $\exists~K_{1b}>0$ such that
  %   \[\limsup_{t\to\infty} \Pr(\vert X\vert > t)\exp\big((t/K_{1b})^q\big)<\infty\]
  % \end{enumerate} 
  \item Growth rate of absolute moments: $\exists~K_2>0$ such that $\forall~p\geq 1$, 
  \[\big(\E[\vert X\vert^p]\big)^{1/p}\leq K_2 p^{1/q}\] %2a
  % \begin{enumerate}[label=(\alph*)]
  %   \item $\exists~K_2>0$ such that $\forall~p\geq 1$,
  %   \[\big(\E[\vert X\vert^p]\big)^{1/p}\leq K_2 p^{1/q}\]
  %   \item \[\limsup_{p\to\infty} \frac{\big(\E[\vert X\vert^p]\big)^{1/p}}{p^{1/q}}<\infty\]
  % \end{enumerate}
  \item MGF of $\vert X\vert^q$ finite in interval of zero: $\exists~K_3>0$ such that $\forall~0<\lambda\leq 1/K_3$
    \[\E\left[\exp(\lambda^q \vert X\vert^q)\right]\leq \exp(K_3^q \lambda^q)\]
  % \begin{enumerate}[label=(\alph*)]
  %   \item $\forall~0<\lambda<\frac{2^{1/q}}{K_3}$
  %   \[\E\left[\exp(\lambda^q \vert X\vert^q)\right]\leq \frac{1}{1-\lambda^q K_3^q/2}\]
  %   \item $\forall~0<\lambda\leq 1/K_3$
  %   \[\E\left[\exp(\lambda^q \vert X\vert^q)\right]\leq \exp(K_3^q \lambda^q)\]
  % \end{enumerate}
  % \item  Finite $\psi_q$-Orlicz norm, $\Vert X\Vert_{\psi_q}\leq K_4<\infty$, so that 
  % \[\E\left[\exp\left\{\left(\frac{\vert X\vert}{K_4}\right)^q\right\}\right]\leq 2\]  
\end{enumerate}
\end{proposition}
%can be omitted
When $q\geq 1$, Condition 3 is equivalent to requiring the Orlicz norm $\Vert X \Vert_{\psi_q}<\infty$ where $\psi_q(x)=\exp(x^q)-1$. The case of $q<1$ (heavy tails) is further discussed in \cite{kuchibhotlaMovingSubGaussianityHighdimensional2022}. 
It is sometimes convenient to use the following asymptotic alternatives to Proposition \ref{prop:subweibull} (1 and 2).

\begin{corollary} %can be omitted
\label{cor:subweibull}
A random variable $X$ is q-subweibull if and only if either of the following hold
\begin{enumerate}
  \item Tail bound: $\exists~K_{1b}>0$ such that %1b
    \[\limsup_{t\to\infty} \Pr(\vert X\vert > t)\exp\big((t/K_{1b})^q\big)<\infty\]
  \item Growth rate of absolute moments:
    \[\limsup_{p\to\infty} \frac{\big(\E[\vert X\vert^p]\big)^{1/p}}{p^{1/q}}<\infty\]
\end{enumerate}
\end{corollary}
\begin{proof}
Proposition \ref{prop:subweibull} $(1)\implies (1)$ :
\[
\limsup_{t\to\infty} \Pr(\vert X\vert > t)\exp\big((t/K_{1a})^q\big)\leq \sup_{t\geq 0} \Pr(\vert X\vert > t)\exp\big((t/K_{1a})^q\big)\leq 2<\infty
\]
So we can simply set $K_{1b}=K_{1a}$.

$(1b)\implies$ Proposition \ref{prop:subweibull} $(1a)$:
Assume 
\[\limsup_{t\to\infty} \Pr(\vert X\vert > t)\exp\big((t/K_{1b})^q\big)=K\]
Then, for every $C>K$, there exists some $T$ such that for all $t>T$, 
\[\Pr(\vert X\vert > t)\exp\big((t/K_{1b})^q\big)<C\]
For all $t\in[0,T]$, $\Pr(\vert X\vert > t)\leq 1$ and $\exp\big((t/K_{1b})^q\big)\leq \exp\big((T/K_{1b})^q\big)$. Therefore 
\[\sup_{t\geq 0} \Pr(\vert X\vert > t)\exp\big((t/K_{1b})^q\big)\leq \max\left\{C,~\exp\big((T/K_{1b})^q\big)\right\}\]
Let $U = \max\left\{C,~\exp\big((T/K_{1b})^q\big)\right\}$. If $U\leq 2$ this directly implies $(1a)$ with $K_{1a}=K_{1b}$. In the case that $U>2$, set 
\[K_{1a}=K_{1b}\left(\frac{\log U}{\log 2}\right)^{1/q}> K_{1b}\]
Let $f(t)= U\exp\big(-(t/K_{1b})^q\big)$, $g(t)=2\exp\big(-(t/K_{1a})^q\big)$, and $T^\star=K_{1b} (\log U)^{1/q}$. Since $f(T^\star)=g(T^\star)=1$ and $g(t)$ is a strictly decreasing function, this implies that $\Pr(\vert X\vert > t)\leq 1\leq g(t)$ for $t\in[0,T^\star]$. For $t\geq T^\star$, $g(t)\geq f(t)$ since $K_{1a}>K_{1b}$, and $f(t)\geq \Pr(\vert X\vert > t)$ by assumption therefore $2\exp\big(-(t/K_{1a})^q\big)\geq \Pr(\vert X\vert > t)$ for all $t\geq 0$.

Proposition \ref{prop:subweibull} $(2)\implies (2)$: 
\[\limsup_{p\to\infty} \frac{\big(\E[\vert X\vert^p]\big)^{1/p}}{p^{1/q}}\leq \sup_{p\geq 1} \frac{\big(\E[\vert X\vert^p]\big)^{1/p}}{p^{1/q}}\leq K_2 < \infty\]

$(2)\implies$ Proposition \ref{prop:subweibull} $(2)$: 
Assume \[\limsup_{p\to\infty} \frac{\big(\E[\vert X\vert^p]\big)^{1/p}}{p^{1/q}}=K<\infty\]
Then for every $C>K$, there exists some $p^\star$ such that for all $p>p^\star$,
\[\frac{\big(\E[\vert X\vert^p]\big)^{1/p}}{p^{1/q}}<C\]
The $L_p$ norm is increasing in $p$, so for $p\in[1,p^\star]$, $\big(\E[\vert X\vert^p]\big)^{1/p}\leq \big(\E[\vert X\vert^{p^\star}]\big)^{1/p^\star}$ and $p^{1/q}\geq 1$, which establishes
\[\sup_{p\geq 1}\frac{\big(\E[\vert X\vert^p]\big)^{1/p}}{p^{1/q}}\leq \max\left\{\big(\E[\vert X\vert^{p^\star}]\big)^{1/p^\star},~C\right\}\]
\end{proof}

It was shown by \cite{vladimirovaSubWeibullDistributionsGeneralizing2020} that a q-subweibull distribution is also r-subweibull for all $r<q$. We now show that this also implies it is strictly r-subweibull.

\begin{corollary}
\label{cor:strict-subweib}
If $X$ is q-subweibull then it is strictly r-subweibull for all $r\in (0,q)$. 
\end{corollary}
\begin{proof}
If $X$ is q-subweibull, by Proposition \ref{prop:subweibull} we may assume there exists $K>0$ such that $\forall~p\geq 1$,
\[\big(\E[\vert X\vert^p]\big)^{1/p}\leq K p^{1/q}\]
Let $r\in (0,q)$. The MGF of $\vert X\vert^r$ is given by
\begin{align*}
\E\left[\exp(\lambda^r \vert X\vert^r)\right] &= 1+\sum_{p=1}^\infty \frac{\lambda^{pr} E[\vert X\vert^{pr}]}{p!}\\
&\leq 1+\sum_{p=1}^\infty \frac{\lambda^{pr}K^{pr} (pr)^{pr/q}}{(p/e)^p}\\ 
&= 1+\sum_{p=1}^\infty \big(\lambda^r K^r e r^{r/q}\big)^p p^{p(r/q-1)}
\end{align*}
Apply the root test to the series to determine convergence.
\begin{align*}
R(p) = \lambda^r K^r e r^{r/q} p^{r/q-1}
\end{align*}
% If $r>q$ then $\lim_{p\to\infty} R(p) = \infty$ and the series diverges. %this means we don't know anything about |X|^r
% If $r=q$ then $R(p)= \lambda^r K^q e q^{(1)}$ for all $p$ and the series converges if $\lambda<K^{-1}(eq)^{-1/q}$.
Since $r<q$, then $\lim_{p\to\infty} R(p) = 0$ and the series converges regardless of the value of $\lambda$, which shows $X$ is strictly r-subweibull.
\end{proof}

\begin{corollary}
\label{cor:bounded-subweib}
Every bounded random variable is strictly q-subweibull for all $q>0$.
\end{corollary}
\begin{proof}
If $X$ is bounded then there exists $M\geq 0$ such that $\vert X\vert \leq M$. Then $\E[\exp(\lambda^q\vert X\vert^q)]\leq \exp(\lambda^q M^q)<\infty$ for all $\lambda>0$ and $q>0$.
\end{proof}

\begin{corollary} %can be omitted
\label{cor:not-subweib}
If $X$ is not strictly q-subweibull with $q\geq 1$ then it is not r-subweibull for any $r>q$.
\end{corollary}
\begin{proof}
From Corollary \ref{cor:subweibull} we may assume $\exists~\lambda>0$ such that 
\[\limsup_{t\to\infty} \Pr(\vert X\vert>t)\exp(\lambda t^q)=\infty\]
which implies there is an infinite sequence $t_n\to\infty$ such that
\[\lim_{n\to\infty} \Pr(\vert X\vert>t_n)\exp(\lambda t_n^q)=\infty\]
Now let $\rho>0$ and $r>q$. Whenever $t\geq t^\star=(\lambda/\rho)^{1/(r-q)}$, we have $\exp(\rho t^r)\geq \exp(\lambda t^q)$. Let $\{t_m\}$ be the infinite subsequence of $\{t_n\}$ excluding the elements less than $t^\star$. Clearly $t_m\to\infty$ as well. Then 
\[\lim_{m\to\infty} \Pr(\vert X\vert > t_m)\exp(\rho t_m^r)\geq \lim_{m\to\infty} \Pr(\vert X\vert > t_m)\exp(\lambda t_m^q) = \infty \]
which implies $X$ cannot be r-subweibull.
\end{proof}

\subsection{Subweibull properties of the Poisson distribution}

%can be omitted
Corollaries \ref{cor:strict-subweib} and \ref{cor:not-subweib} suggest a hierarchy of distributions based on the heaviness of the tails. Broadly q-subweibull distributions, which have a finite but nonzero radius of convergence ($R_q$), serve as ``critical points'' in the transition between the strictly r-subweibull regime ($r<q$), with $R_q=\infty$ and the not r-subweibull regime ($r>q$) with $R_q=0$. However, the transition from strictly subweibull to not subweibull can be immediate, without passing through the stage of broadly subweibull. A sophisticated example using the moment sequence was suggested by \cite{4907906}. Here we provide a simple example: the Poisson tail is lighter than any exponential tail, but heavier than any weibull tail with $q>1$. 
\begin{proposition}
\label{prop:poi-subweib}
The Poisson distribution is strictly q-subweibull for $q\leq 1$ but not q-subweibull for any $q>1$.
\end{proposition}
\begin{proof}
Since the Poisson distribution has a finite MGF with infinite radius of convergence, it is strictly subexponential and by Corollary \ref{cor:strict-subweib} strictly q-subweibull for all $q\leq 1$. Let $X\sim Poi(\mu)$. Without loss of generality assume $t>1$ and let $n=\lfloor t\rfloor +1$ with $t<n\leq t+1$. 
\[\Pr(X>t) = \sum_{j=n}^\infty \Pr(X=j) \geq \Pr(X=n)=\frac{\mu^{n}\exp(-\mu)}{n!} = \frac{\mu^{n}\exp(-\mu)}{n\Gamma(n)}\]
Since $t\Gamma(t)$ is increasing for $t\geq 1$, we have $n\Gamma(n)\leq (t+1)\Gamma(t+1)$. Also, $\Gamma(n)\leq n^n$ for $n\geq 1$. For the $\mu^n$ term, it is increasing for $\mu\geq 1$ and decreasing for $\mu<1$, so $\mu^n\geq \min\{\mu^{t+1},\mu^t\}=\mu^t\min\{\mu,1\}$. Combining these we obtain
\[\Pr(X>t)\geq \frac{\mu^t\min\{\mu,1\}e^{-\mu}}{(t+1)\Gamma(t+1)} = \frac{\mu^t \min\{\mu,1\}e^{-\mu}}{(t+1)(t)\Gamma(t)}\geq \frac{\mu^t \min\{\mu,1\}e^{-\mu}}{(t+1)(t)t^t}\]
To assess whether the tail follows a subweibull rate of decay, choose any $\lambda>0$ and $q>1$, then
\begin{align*}
\limsup_{t\to\infty}& \Pr(X>t)\exp(\lambda t^q)\geq \min\{\mu,1\}e^{-\mu}\lim_{t\to\infty}\frac{\mu^t}{(t+1)(t)t^t}\exp(\lambda t^q)\\
&= \min\{\mu,1\}e^{-\mu}\exp\left[\lim_{t\to\infty}t\log\mu-\log(t+1)-\log t-t\log t+\lambda t^q\right]
\end{align*}
The expression inside brackets is of the form $\infty-\infty$ so we rearrange terms and apply L'Hopital's rule. Define
\begin{align*}
\lim_{t\to\infty}&t\log\mu-\log(t+1)-\log t-t\log t+\lambda t^q\\
&= \lim_{t\to\infty} (t\log t)\left[\lim_{t\to\infty}\frac{\log\mu}{\log t} - \frac{\log(t+1)}{t\log t}-\frac{1}{t}+\frac{\lambda t^q}{t\log t}\right]\\
&= \lim_{t\to\infty} (t\log t)\left[\lim_{t\to\infty}0 - \frac{1/(t+1)}{1+\log t}-(0)+\frac{\lambda qt^{q-1}}{1+\log t}\right]\\
&= \lim_{t\to\infty} (t\log t)\left[\lim_{t\to\infty}\frac{\lambda q(q-1)t^{q-2}}{1/t}\right] = \infty \cdot \infty = \infty
\end{align*}
Therefore
\[\limsup_{t\to\infty} \Pr(X>t)\exp(\lambda t^q)= \infty\]
Since this holds for all $\lambda>0$, $X$ cannot satisfy Corollary \ref{cor:subweibull} and therefore is not q-subweibull for any $q>1$.
\end{proof}

\section{Exponential tilting}
\begin{definition}
\label{def:exp_tilt}
Let $X$ be a random variable with distribution function $F$. If the Laplace transform satisfies $\mathcal{L}_X(-\theta)<\infty$ for some $\theta\neq 0$, then the \textit{exponentially tilted distribution} is given by
\[F_\theta(x) = \int_{-\infty}^x \frac{\exp(\theta t)}{\mathcal{L}_X(-\theta)}dF(t)\]
\end{definition}
We adopt the convention of using $-\theta$ instead of $\theta$ so that the interpretation of the tilting parameter is consistent with other works that assume $X$ has an MGF, in which case one could equivalently require $M_X(\theta)<\infty$. 

From the Radon-Nikodym theorem, $F_\theta$ is absolutely continuous with respect to $F$. Since the density function $e^{\theta x}/\mathcal{L}_X(-\theta)$ is also strictly positive, exponential tilting does not change the support. Generally speaking it is possible to produce a subexponential distribution by exponential tilting of any distribution with at least one light tail.

\begin{proposition}
\label{prop:exp-tilt-general-exp}
If $X\sim F$ is a random variable having at least one light tail then exponential tilting is possible for all $\theta$ in some open interval $(-S,T)$ with $S,T\geq 0$ and $S+T>0$. The resulting tilted distribution $F_\theta$ is subexponential with MGF $M_Z(t)=\mathcal{L}_X(-\theta-t)/\mathcal{L}_X(-\theta)$ finite for all $t\in (-S-\theta,T-\theta)$.
\end{proposition}
\begin{proof}
Without loss of generality assume the right tail is light so $\mathcal{L}_X(-\theta)<\infty$ for some $\theta>0$. For all $\theta'\in [0,\theta)$, 
\[\mathcal{L}_X(-\theta')=\E[\exp(\theta'X)] \leq \E[\exp(\theta X)]<\infty\]
Set $T=\sup\{\theta:~\mathcal{L}_X(-\theta)<\infty\}>0$. If $X$ has a heavy left tail then $\mathcal{L}_X(-\theta)=\infty$ for all $\theta<0$, so the interval of convergence is $(-S,T)$ with $S=0$. If $X$ has a light left tail then we can set $S=-\inf\{\theta:\mathcal{L}_X(-\theta)<\infty\}>0$. This establishes the interval is $(-S,T)$ with $S,T\geq 0$ and $S+T>0$. Let $Z\sim F_\theta$ follow the tilted distribution with $\theta\in (-S,T)$. Its Laplace transform is
\begin{align*}
\mathcal{L}_Z(t)&=\E[\exp(-tZ)] = \int_{-\infty}^\infty \exp(-tz)dF_\theta(z) = \int_{-\infty}^\infty \exp(-tx)\frac{\exp(\theta x)}{\mathcal{L}_X(-\theta)}dF(x) \\
&= \E[\exp(-(t-\theta)X)]/\mathcal{L}_X(-\theta) = \mathcal{L}_X(-(\theta-t))/\mathcal{L}_X(-\theta)
\end{align*}
This is finite when $\theta-t \in (-S,T)$ or equivalently $t \in (-T+\theta, S+\theta)$. Since $\theta\in (-S,T)$, the interval of convergence for $\mathcal{L}_Z(t)$ is an open interval containing zero, which proves $Z$ is subexponential and has the MGF
\[M_Z(t)=\mathcal{L}_Z(-t)=\mathcal{L}_X(-\theta-t)/\mathcal{L}_X(-\theta)\]
which is finite on the interval $t\in (-S-\theta, T-\theta)$. 
\end{proof}

As an example, if $X\sim F$ is a nonnegative, heavy tailed random variable ($T=0$), its left tail is strictly subexponential ($S=\infty$) so exponential tilting is possible for all $\theta<0$. By Proposition \ref{prop:exp-tilt-general-exp} the resulting tilted distribution is subexponential and hence has lighter tails than the original distribution. On the other hand, if $X$ is broadly subexponential, exponential tilting produces another broadly subexponential distribution, with a shifted interval of convergence.

While exponential tilting can alter the tail behavior of heavy tailed and broadly subexponential distributions, it does not affect the tail behavior of q-subweibull distributions with lighter than exponential tails (i.e., $q>1$).

\begin{lemma}
\label{lem:exp-tilt-nonneg}
Preservation of nonnegative subweibull tails under exponential tilting. Let $\theta$ be any real number.
% \begin{enumerate}
If $X\sim F$ is nonnegative and q-subweibull ($q>1$), then the exponentially tilted variable $Z\sim F_\theta$ is also nonnegative and q-subweibull with the same radius of convergence.
\begin{enumerate}%[label=(\alph*)]
  \item $\E[\exp(\lambda^q X^q)]<\infty$ for all $\lambda\in [0,R_q)$ implies $\E[\exp(\lambda^q Z^q)]<\infty$ for all $\lambda\in [0,R_q)$.
  \item $\E[\exp(\lambda^q X^q)]=\infty$ for all $\lambda>R_q$ implies $\E[\exp(\lambda^q Z^q)]=\infty$ for all $\lambda>R_q$. 
\end{enumerate}
 % \item If $X\sim F$ is strictly q-subweibull ($q\geq1$), the exponentially tilted variable $Z\sim F_\theta$ is also strictly q-subweibull.
 % \item If $X\sim F$ is not q-subweibull ($q>1$), then $Z\sim F_\theta$ is also not q-subweibull.
% \end{enumerate}
\end{lemma}
\begin{proof}
If $X$ is q-subweibull with $q>1$ then by Corollary \ref{cor:strict-subweib} it is strictly subexponential and $\mathcal{L}_X(-\theta)<\infty$ for all $\theta\in\mathbb{R}$. Let $Z\sim F_\theta$. The MGF of $Z^q$ is
\[\E[\exp(\lambda^q Z^q)] = \frac{\int \exp(\lambda^q x^q + \theta x) dF(x)}{\mathcal{L}_X(-\theta)}\]

$(1)$ case of $\lambda<R_q$. If $\theta\leq 0$ then 
\[\int \exp(\lambda^q x^q + \theta x) dF(x)\leq \int \exp(\lambda^q x^q + 0) dF(x) = \E[\exp(\lambda^q X^q)]<\infty\]
If $\theta>0$, choose $\rho\in (\lambda,R_q)$ and define
\[x^\star = \left(\frac{\theta}{\rho^q-\lambda^q}\right)^{\frac{1}{q-1}}\]
Then for $x>x^\star$, $\lambda^q x^q + \theta x\leq \rho^q x^q$. Therefore,
\begin{align*}
\int \exp(\lambda^q x^q + \theta x) dF(x) &= \int_0^{x^\star} \exp(\lambda^q x^q + \theta x) dF(x) + \int_{x^\star}^\infty \exp(\lambda^q x^q + \theta x) dF(x)\\
&\leq \int_0^{x^\star}\exp\big(\theta x^\star + \lambda^q (x^\star)^q\big)dF(x) + \int_{x^\star}^\infty \exp(\rho^q x^q) dF(x)\\
&\leq \exp\big(\theta x^\star + \lambda^q (x^\star)^q\big) \Pr(X\leq x^\star) + \int_0^\infty \exp(\rho^q x^q) dF(x)\\
&<\infty
\end{align*}

$(2)$ case of $\lambda>R_q$. If $\theta\geq 0$ then
\[\int \exp(\lambda^q x^q + \theta x) dF(x)\geq \int \exp(\lambda^q x^q + 0) dF(x) = \E[\exp(\lambda^q X^q)] = \infty\]
If $\theta<0$. Choose $\rho\in (R_q,\lambda)$ and define
\[x^\star = \left(\frac{-\theta}{\lambda^q-\rho^q}\right)^{\frac{1}{q-1}}\]
Then for $x>x^\star$, $\lambda^q x^q + \theta x\geq \rho^q x^q$. Therefore,
\begin{align*}
\int \exp(\lambda^q x^q + \theta x) dF(x) &= \int_0^{x^\star} \exp(\lambda^q x^q + \theta x) dF(x) + \int_{x^\star}^\infty \exp(\lambda^q x^q + \theta x) dF(x)\\
&\geq \int_0^{x^\star}\exp(\lambda^q x^q + \theta x)dF(x) + \int_{x^\star}^\infty \exp(\rho^q x^q) dF(x)
\end{align*}
The first term is finite. We will show the second term is infinite. By assumption,
\begin{align*}
\int_0^\infty \exp(\rho^q x^q) dF(x) &= \infty\\
&= \int_0^{x^\star} \exp(\rho^q x^q) dF(x)+\int_{x^\star}^\infty \exp(\rho^q x^q) dF(x)
\end{align*}
But 
\[\int_0^{x^\star} \exp(\rho^q x^q) dF(x)\leq \exp\big(\rho^q (x^\star)^q\big)\Pr(X\leq x^\star)<\infty\]
Therefore
\[\int_{x^\star}^\infty \exp(\rho^q x^q) dF(x) = \infty\]
implying
\[\int \exp(\lambda^q x^q + \theta x) dF(x)=\infty\]
as well.

% $(2)$ For $q=1$ apply Proposition \ref{prop:exp-tilt-general-exp} with $R=\infty$ and $S=\infty$. For $q>1$ apply $(1a)$ with $R_q=\infty$. 

% $(3)$ This is a direct corollary of $(1b)$ obtained in the case of $R_q=0$.
\end{proof}

We now extend Lemma \ref{lem:exp-tilt-nonneg} to general random variables. 

\begin{theorem}
\label{thm:exp-tilt-general}
Preservation of subweibull tails under exponential tilting. Let $\theta$ be any real number.
\begin{enumerate}
 \item If $X\sim F$ is q-subweibull ($q>1$) with radius of convergence $R_q$, then the exponentially tilted variable $Z\sim F_\theta$ is also q-subweibull and has the same radius of convergence.
   % \begin{enumerate}[label=(\alph*)]
   % \item $\E[\exp(\lambda^q \vert X\vert^q)]<\infty$ for all $\lambda\in [0,R_q)$ implies $\E[\exp(\lambda^q \vert Z\vert^q)]<\infty$ for all $\lambda\in [0,R_q)$.
   % \item $\E[\exp(\lambda^q \vert X\vert^q)]=\infty$ for all $\lambda>R_q$ implies $\E[\exp(\lambda^q \vert Z\vert^q)]=\infty$ for all $\lambda>R_q$. 
   % \end{enumerate}
 \item If $X\sim F$ is strictly q-subweibull ($q\geq1$), the exponentially tilted variable $Z\sim F_\theta$ is also strictly q-subweibull.
 \item If $X\sim F$ is not q-subweibull ($q>1$), then $Z\sim F_\theta$ is also not q-subweibull.
\end{enumerate}
\end{theorem}
\begin{proof}
$(1)$ By Corollary \ref{cor:strict-subweib}, $X$ is strictly subexponential so $\mathcal{L}_X(-\theta)<\infty$ for all $\theta\in\mathbb{R}$. Choose any arbitrary $\theta$ and set $M_1=\mathcal{L}_X(-\theta)$. Define nonnegative random variables $A=[X\given X<0]$ and $B=[X\given X\geq 0]$ with distributions $F^-$ and $F^+$, respectively. By Lemma \ref{lem:subweibull-split} both $A$ and $B$ are q-subweibull and strictly subexponential. Let $R_{qa}$ and $R_{qb}$ be the radii of convergence of $A$ and $B$, respectively. Let $p=\Pr(X<0)$ and assume $p\notin \{0,1\}$ (otherwise simply apply Lemma \ref{lem:exp-tilt-nonneg} to $X$ or $-X$). Note that 
\begin{equation}
\label{eq:decompose-laplace}
M_1 = \mathcal{L}_A(\theta)p + \mathcal{L}_B(-\theta)(1-p)
\end{equation}
% \E[\exp(
We have
\begin{align*}
\E[\exp(\lambda^q\vert Z\vert^q)] &= \int_{-\infty}^\infty \exp(\lambda^q \vert z\vert^q)dF_\theta(z) = \int_{-\infty}^\infty \exp(\lambda^q\vert x\vert^q)\frac{\exp(\theta x)}{M_1}dF(x)\\
&= \frac{\E[\exp(\lambda^q\vert X\vert^q + \theta X)]}{M_1}\\
&= \frac{p}{M_1}\E[\exp(\lambda^q A^q - \theta A)] + \frac{(1-p)}{M_1}\E[\exp(\lambda^q B^q + \theta B)]\\
&= \frac{p \mathcal{L}_{A}(\theta)}{M_1}\int_0^\infty \exp(\lambda^q x^q)\frac{\exp(-\theta x)}{\mathcal{L}_{A}(\theta)}dF^-(x) \ldots\\
&~\ldots+ \frac{(1-p)\mathcal{L}_{B}(-\theta)}{\mathcal{L}_X(-\theta)}\int_0^\infty \exp(\lambda^q x^q)\frac{\exp(\theta x)}{\mathcal{L}_{B}(-\theta)}dF^+(x)\\
&= (\tilde{p}) \int_0^\infty \exp(\lambda^q z^q)dF_{(-\theta)}^-(z) + (1-\tilde{p}) \int_0^\infty \exp(\lambda^q z^q)dF_{\theta}^+(z)\\
&= \tilde{p} \E[\exp(\lambda^q U^q)]+(1-\tilde{p})\E[\exp(\lambda^q V^q)]
\end{align*}
where (see Equation \ref{eq:decompose-laplace}), $\tilde{p}=p\mathcal{L}_A(\theta)/M_1$ so that $\tilde{p}\in (0,1)$. The nonnegative random variable $U$ is distributed as $F_{(-\theta)}^-$, which is the exponentially tilting of $A~F^-$ by $-\theta$ and $V\sim F_\theta^+$ is similarly defined as the exponential tilting of $B\sim F^+$. By Lemma \ref{lem:exp-tilt-nonneg}, this implies $U$ and $V$ are q-subweibull with radii of convergence $R_{qa}$ and $R_{qb}$, respectively. Let $R_{qz}$ be the radius of convergence of $Z$. Note that 
\begin{align*}
\Pr(Z\geq 0) &= \int_0^\infty dF_\theta(z) = \int_0^\infty \frac{\exp(\theta x)}{\mathcal{L}_X(-\theta)}dF(x) = \frac{\E[\exp(\theta X)\given X\geq 0]\Pr(X\geq 0)}{M_1}\\
&= \frac{\E[\exp(\theta B)](1-p)}{M_1} = \frac{\mathcal{L}_B(-\theta) (1-p)}{M_1}\\
&= 1-\tilde{p}
\end{align*}
So $\Pr(Z<0)=\tilde{p}$. By Lemma \ref{lem:subweibull-split} this implies $Z$ is q-subweibull with radius of convergence $\min\{R_{qa},R_{qb}\}$, which is also the radius of convergence of $X$. 

$(2)$ For $q=1$ apply Proposition \ref{prop:exp-tilt-general-exp} with $S=\infty$ and $T=\infty$. For $q>1$, apply $(1)$ with $R_q=\infty$. 

$(3)$ This is a direct corollary of $(1)$ obtained in the case of $R_q=0$.
\end{proof}

\section{Discussion}

The theory of subexponential and subgaussian distributions is a key prerequisite to many results in nonparametric and nonasymptotic statistical inference such as concentration inequalities. A comprehensive overview with applications to high-dimensional machine learning problems is provided by \cite{kuchibhotlaMovingSubGaussianityHighdimensional2022}. Exponential tilting is used in a variety of statistical areas such as causal inference \cite{mccleanNonparametricEstimationConditional2024} and Monte Carlo sampling \cite{fuhEfficientExponentialTilting2024}. If $F_\theta$ is a tilted distribution, it is a natural exponential family with parameter $\theta$. The exponential families are building blocks for generalized linear models \cite{mccullaghGeneralizedLinearModels1989}. For applications to machine learning see \cite{liTiltedLossesMachine2023,maityUnderstandingNewTasks2023}.

Here, we have provided a brief overview of subweibull distributions. We showed that the Poisson distribution is strictly 1-subweibull but not q-subweibull for any $q>1$. Finally, we detailed the conditions under which the subweibull property is perserved under exponential tilting. 

\section*{Acknowledgements}
Thanks to Arun Kumar Kuchibhotla, Sam Power, Patrick Staples, Valerie Ventura, Larry Wasserman, and Matt Werenski for helpful comments and suggestions.
%Dan Rowe,

% \bibliography{/Users/townesf/documents/research/zotero_library}
\bibliography{refs}

\end{document}